\documentclass{amsart}

\usepackage{amsfonts}
\usepackage{amsmath}
\usepackage{graphicx}

\usepackage{hyperref}
\usepackage{graphicx}

\newtheorem{theorem}{Theorem}
\theoremstyle{plain}

\newtheorem{corollary}{Corollary}

\numberwithin{equation}{section}

\begin{document}

\title[Sharp Wirtinger's type inequalities for double integrals ]{Sharp Wirtinger's type inequalities for double integrals with Applications}
\author[M.W. Alomari]{Mohammad W. Alomari}
\address{Department of Mathematics,
Faculty of Science and Information Technology, Irbid National
University, 21110 Irbid, Jordan} \email{mwomath@gmail.com}


\date{\today}
\subjclass[2000]{Primary 26D15 ; Secondary 26B30, 26B35.}

\keywords{Wirtinger's inequality, \v{C}eby\v{s}ev functional,
Gr\"{u}ss inequality}

\begin{abstract}
In this work, sharp Wirtinger type inequalities for double
integrals are established. As applications, two sharp
\v{C}eby\v{s}ev type inequalities for absolutely continuous
functions whose second partial derivatives belong to $L^2$ space
are proved.
\end{abstract}

\maketitle

\section{Introduction}
The theory of Fourier series has a significant role in almost all
branches of mathematical and numerical analysis. A very
interesting connection between inequalities and Fourier series has
been made along more than a hundred year ago. The celebrated
Bessel's integral inequality
\begin{align}
2a_0^2  + \sum\limits_{n = 1}^\infty {\left( {a_n^2  + b_n^2 }
\right)}  \le \frac{1}{\pi }\int_{ - \pi }^\pi  {f^2 \left( x
\right)dx},
\end{align}
was named after Bessel death and considered from that time as the
first adobe in this connection and started point for other related
works after the end of 18-th century.

In 1916, Wirtinger \cite{Blaschke} credibly proved his inequality
regarding square integrable periodic functions, which reads:
\begin{theorem}
\label{W.thm}Let $f$ be a real valued function with period $2\pi$
and $\int_0^{2\pi}f\left({x}\right)dx=0$. If $f' \in L^2[0,2\pi]$,
then
\begin{align}
\label{eq1.1}\int_0^{2\pi } {f^2 \left( x \right)dx}  \le
\int_0^{2\pi } {f'^2 \left( x \right)dx},
\end{align}
with equality if and only if $f(x)=A\cos x + B \sin x$, $A,B\in
\mathbb{R}$.
\end{theorem}
Many authors have considered a main attention for Wirtinger's
inequality and therefore, several generalizations, counterparts
and refinements was collected in a chapter of the book
\cite{Mitrinvovic}.

In 1967, Diaz and Metcalf \cite{Diaz} have extended and
generalized Wirtinger inequality and they proved the following
result:
\begin{theorem}
\label{Diaz.thm}Let $f$ be continuously differentiable on $(a,b)$.
Suppose $f(t_1)=f(t_2)$ for $a\le t_1 \le t_2 \le b$, then the
inequality
\begin{multline}
\label{eq1.3}\int_a^{b} {\left[f \left( x \right)-f \left( t_1
\right)\right]^2 dx}
\\
\le \frac{4}{\pi^2} \max \left\{ {\left( {t_1 - a} \right)^2
,\left( {b - t_2 } \right)^2 ,\left( {\frac{{t_2  - t_1 }}{2}}
\right)^2 } \right\} \int_a^{b } {f'^2 \left( x \right)dx},
\end{multline}
holds. In particular, if $t_1=t_2=t$, then
\begin{align}
\label{eq1.4}\int_a^{b} {\left[f \left( x \right)-f \left( t
\right)\right]^2 dx} \le \frac{4}{\pi^2}  \left[ {\frac{{b -
a}}{2} + \left| {t  - \frac{{a + b}}{2}} \right|} \right]^2
\int_a^{b } {f'^2 \left( x \right)dx},
\end{align}
\end{theorem}
For other related results see \cite{Beesack1}, \cite{Beesack2} and
\cite{Milovanovic}.

One of the most direct applicable usage of (\ref{eq1.3}) were
considered in several works regarding the famous
\emph{\v{C}eby\v{s}ev functional}
\begin{align}
\label{identity}\mathcal{T}\left( {f,g} \right) = \frac{1}{{b -
a}}\int_a^b {f\left( t \right)g\left( t \right)dt}  - \frac{1}{{b
- a}}\int_a^b {f\left( t \right)dt}  \cdot \frac{1}{{b -
a}}\int_a^b {g\left( t \right)dt}.
\end{align}
which compare or measure the difference between the integral
product of two functions with their corresponding integrals
product.

In 1970, Ostrowski \cite{Ostrowski}  proved that  if $f', g' \in
L^2[a, b]$, then there exists a constant $C$, $0\le C \le
\frac{b-a}{8}$, such that
\begin{equation}
\label{Ostrowski}\left| {\mathcal{T}\left( {f,g} \right)} \right|
\le C\left\| {f'} \right\|_2 \left\| {g'} \right\|_2.
\end{equation}
After that in 1973, A. Lupa\c{s} \cite{Lupas} has improved the
result of Ostrowski's (\ref{Ostrowski}) and proved that
\begin{equation}
\label{eq.Lupas}\left| {\mathcal{T}\left( {f,g} \right)} \right|
\le \frac{b-a}{{\pi}^2}\left\| {f'} \right\|_2 \left\| {g'}
\right\|_2.
\end{equation}
where, the constant $\frac{1}{\pi^2}$ is the best possible.

In this work we deal with the problem: what is the best possible
constant $C$ would the inequality
\begin{align}
\int_c^d{\int_a^b {f^2\left( {x,y} \right)dx}dy} \le C
\int_c^d{\int_a^b {\left( {\frac{{\partial ^2 f}}{{\partial
x\partial y}}} \right)^2}dxdy}
\end{align}
holds, whenever $f,g \in \mathfrak{L}^2(I)$. This question is a
natural extension of Diaz-Metcalf inequality (\ref{eq1.3}), as
well as the complementary works of Beesack and Milovanovi\'{c} in
one variable, see  \cite{Beesack1}, \cite{Beesack2} and
\cite{Milovanovic}.

Accordingly, for the \emph{\v{C}eby\v{s}ev functional}
\begin{align*}
\mathcal{T}\left( {f,g} \right) &:= \frac{1}{{\left( {b - a}
\right)\left( {d - c} \right)}}\int_a^b {\int_c^d {f\left( {t,s}
\right)g\left( {t,s} \right)dsdt} }
\nonumber\\
&\qquad- \frac{1}{{\left( {b - a} \right)\left( {d - c}
\right)}}\int_a^b {\int_c^d {f\left( {t,s} \right)dsdt} }
\frac{1}{{\left( {b - a} \right)\left( {d - c} \right)}}\int_a^b
{\int_c^d {g\left( {t,s} \right)dtds} },
\end{align*}
what is the best possible constant $C'$ would the inequality
\begin{align*}
\left| {\mathcal{T}\left( {f,g} \right)} \right| \le
C'\left\|{\frac{{\partial ^2 f}}{{\partial x\partial
y}}}\right\|_2\left\|{\frac{{\partial ^2 g}}{{\partial x\partial
y}}}\right\|_2
\end{align*}
holds, and this is an extension of the Lupa\c{s} inequality
(\ref{eq.Lupas}).

\section{Wirtinger's type inequalities} \label{sec2}
Let $I$ be a two dimensional interval and denotes $I^{\circ}$ its
interior, for $a,b,c,d \in \mathbb{R}$, we consider the subset
$\mathbb{D}:= \left\{ {\left( {x,y} \right):a \le x \le b,c \le y
\le d} \right\}\subseteq \mathbb{R}^2$ such that $\mathbb{D}
\subset I^{\circ}$. Also, define the subsets $I_-$ and $I^-$ of
$I$ as follows:
\begin{align*}
I^ -  : = I - \left\{ {b,d} \right\} = \left[ {a,b} \right) \times
\left[ {c,d} \right),\,\,\,\,{\rm{and}}\,\,\,\,\,I_ -  : = I -
\left\{ {a,c} \right\} = \left( {a,b} \right] \times \left( {c,d}
\right]
\end{align*}
Throughout of this section we assume that $f: I \to \mathbb{R}$
satisfies the boundary conditions: $f(a,\cdot)=f(\cdot,c)=0$,
$f_x(a,\cdot)=f_x(\cdot,c)=0$, $f_y(a,\cdot)=f_y(\cdot,c)=0$ on
$I^{-}$. Also, we assume $f(b,\cdot)=f(\cdot,d)=0$,
$f_x(b,\cdot)=f_x(\cdot,d)=0$, $f_y(b,\cdot)=f_y(\cdot,d)=0$ on
$I_{-}$, and both conditions on $I^{\circ}$.

Let $\mathfrak{L}^2(I)$ be the space of all functions $f$ which
are absolutely continuous on $I$, with $ \int_c^d\int_a^b {\left|
{\frac{{\partial ^2 f}}{{\partial x\partial y}}} \right|^2 dxdy} <
\infty $.

\begin{theorem}
\label{thm1}Let $f\in \mathfrak{L}^2(I^{-})$. Then the inequality
\begin{align}
\label{eq.thm1} \int_c^d{\int_a^b {f^2\left( {x,y} \right)dx}dy}
\le \frac{{16}}{{\pi ^4 }}\left( {b - a} \right)^2 \left( {d - c}
\right)^2  \int_c^d{\int_a^b {\left( {\frac{{\partial ^2
f}}{{\partial x\partial y}}} \right)^2}dxdy}
\end{align}
is valid. The constant $\frac{16}{\pi^2}$ is the best possible, in
the sense that it cannot be replaced by a smaller one.
\end{theorem}

\begin{proof}
Let $a \le x <b$ and $c \le y <d$, since $f$ is absolutely
continuous then we can write $f\left( {x,y} \right) =
\int_c^y\int_a^x {f_{ts}\left({t,s} \right)dtds}$. If $a$ and $c$
are real numbers this is equivalent to saying that $f(a,c) = 0$
and $f$ is absolutely continuous on $[a,b)\times [c,d)$. Setting
$$f(x,y)=g_1(x)g_2(y)h(x,y)$$ where, $$g_1\left( x \right) = \sin
\omega_1 \left( {x - a} \right), \forall x\in [a,b)$$ with $
\omega_1 = \lambda _1^{1/2}$ and $\lambda _1  = \frac{{\pi^2 }}{{4
\left( {b - a} \right)^2 }}$, and $$g_2\left( y \right) = \sin
\omega_2 \left( {y - c} \right), \forall y\in [c,d)$$ with $
\omega_2 = \lambda _2^{1/2}$ and $\lambda _2  = \frac{{\pi^2 }}{{4
\left( {d - c} \right)^2 }}$.

\noindent Firstly, let us observe that since $ g'_1\left( x
\right) = \omega_1 \cos \omega_1 \left( {x - a} \right), $ and so
that

\noindent $ g''_1 \left( x \right) =  - \omega_1 ^2 g_1 \left( x
\right). $ Similarly, we have $ g''_2 \left( y \right) =  -
\omega_2 ^{2} g_2 \left( y \right) $.

For simplicity, since $$\frac{{\partial f}}{{\partial x}} = g_1
g_2 \frac{{\partial h}}{{\partial x}} + g_2 hg'_1, $$ then
\begin{align*}
\frac{{\partial ^2 f}}{{\partial x\partial y}} &= g_1 g_2
\frac{{\partial ^2 h}}{{\partial x\partial y}} + g_1 g'_2
\frac{{\partial h}}{{\partial x}} + g'_1 g'_2 h + g'_1 g_2
\frac{{\partial h}}{{\partial y}}
\\
&= \frac{d}{{dy}}\left( {g_1 g_2 \frac{{\partial h}}{{\partial x}}
+ g'_1 g_2 h} \right)
\\
&= g'_2 \left( {g_1 \frac{{\partial h}}{{\partial x}} + g'_1 h}
\right) + g_2 \frac{d}{{dy}}\left( {g_1 \frac{{\partial
h}}{{\partial x}} + g'_1 h}
  \right).
 \end{align*}

Setting $$\Phi : = \Phi \left( {x,y} \right) = g_1 \frac{{\partial
h}}{{\partial x}} + g'_1 h = \frac{{f_x }}{{g_2 }} \Rightarrow g_2
\Phi  = f_x,$$ therefore $$\frac{d}{dy}\left( {g_1 g_2
\frac{{\partial h}}{{\partial x}} + g'_1 g_2 h} \right) = g'_2
\left( {g_1 \frac{{\partial h}}{{\partial x}} + g'_1 h} \right) +
g_2 \left( {g_1 \frac{{\partial h}}{{\partial x}} + g'_1 h}
\right)^\prime
 = \Phi g'_2  + g_2 \Phi_y '.$$

Now, if $a < \alpha < \beta < b$, and \,$c < \gamma < \delta < d$,
we have
\begin{align}
\int_\alpha ^\beta {\int_\gamma ^\delta   {\left( {\frac{{\partial
^2 f}}{{\partial x\partial y}}} \right)^2 dydx} }
 &=  \int_\alpha ^\beta{ \int_\gamma ^\delta  {\left(
{\Phi g'_2  + g_2 \Phi_y '} \right)^2 dydx} }
\nonumber\\
&= \int_\alpha ^\beta{ \int_\gamma ^\delta   {\left( {\Phi g'_2 }
\right)^2 \left( {1 + \frac{{g_2 \Phi_y '}}{{\Phi g'_2 }}}
\right)^2 dydx} }
\nonumber\\
&\ge   \int_\alpha ^\beta{\int_\gamma ^\delta  {\left( {\Phi g'_2
} \right)^2 \left( {1 + 2\frac{{g_2 \Phi_y '}}{{\Phi g'_2 }}}
\right)dydx} }
\nonumber\\
&=  \int_\alpha ^\beta {\int_\gamma ^\delta  {\left( {\Phi g'_2 }
\right)^2 dydx} } + 2 \int_\alpha ^\beta {\int_\gamma ^\delta
{\left( {\Phi g'_2 } \right) g_2 \Phi_y 'dydx} }
\nonumber\\
&= \left. {\left. {g_2 \left( {g'_2 } \right) \Phi ^2 }
\right|_\gamma ^\delta  } \right|_\alpha ^\beta   - \int_\alpha
^\beta  {\int_\gamma ^\delta  {\left( {g_2 \left( {g''_2 } \right)
+(g'_2)^2} \right) \Phi ^2 dydx} }
\nonumber\\
&\qquad+ \int_\alpha ^\beta  {\int_\gamma ^\delta  {\left( {\Phi
g'_2 } \right)^2 dydx} }
\nonumber\\
&= \left. {\left. {g_2 \left( {g'_2 } \right)\Phi ^2 }
\right|_\gamma ^\delta  } \right|_\alpha ^\beta   - \int_\alpha
^\beta  {\int_\gamma ^\delta  {\left( { - \omega _2^2 g_2^2  +
\left( {g'_2 } \right)^2 } \right)\Phi ^2 dydx} } \label{eq2.4}
\\
&\qquad+ \int_\alpha ^\beta  {\int_\gamma ^\delta  {\left( {\Phi
g'_2 } \right)^2 dydx} }
\nonumber\\
&= \left. {\left. {g_2 \left( {g'_2 } \right)\Phi ^2 }
\right|_\gamma ^\delta  } \right|_\alpha ^\beta   + \omega _2^2
\int_\alpha ^\beta  {\int_\gamma ^\delta  {g_2^2 \Phi ^2 dydx} }
\nonumber\\
&\qquad- \int_\alpha ^\beta  {\int_\gamma ^\delta  {\left( {g'_2 }
\right)^2 \Phi ^2 dydx} }  + \int_\alpha ^\beta {\int_\gamma
^\delta  {\left( {\Phi g'_2 } \right)^2 dydx} }
\nonumber\\
&= \left. {\left. {g_2 \left( {g'_2 } \right)\Phi ^2 }
\right|_\gamma ^\delta  } \right|_\alpha ^\beta   + \omega _2^2
\int_\alpha ^\beta  {\int_\gamma ^\delta  {g_2^2 \Phi ^2 dydx} }
\nonumber\\
&= \left. {\left. {g_2 \left( {g'_2 } \right) \Phi ^2 }
\right|_\gamma ^\delta  } \right|_\alpha ^\beta   + \omega _2^2
\int_\alpha ^\beta  {\int_\gamma ^\delta  {\left( {\frac{{\partial
f}}{{\partial x}}} \right)^2 dydx} }\label{eq2.3}
\end{align}
where, in (\ref{eq2.4}) we integrate by parts, assuming that
$u=g_2\left( {g'_2} \right) $ and $dv = 2\Phi_y' \Phi$. Now, we
also have
\begin{align}
\int_\alpha ^\beta  {\int_\gamma ^\delta  {\left( {\frac{{\partial
f}}{{\partial x}}} \right)^2 dydx} }  &= \int_\alpha ^\beta
{\int_\gamma ^\delta  {\left\{ {\left( {g_1 g_2 \frac{{\partial
h}}{{\partial x}} + g'_1 g_2 h} \right)} \right\}^2 dydx} }
\nonumber\\
&= \int_\alpha ^\beta  {\int_\gamma ^\delta  {\left\{ {\left(
{g'_1 g_2 h} \right)^2 \left( {1 + \frac{{g_1 g_2 \frac{{\partial
h}}{{\partial x}}}}{{g'_1 g_2 h}}} \right)^2 } \right\}dydx} }
\nonumber\\
&\ge \int_\alpha ^\beta  {\int_\gamma ^\delta  {\left\{ {\left(
{g'_1 g_2 h} \right)^2 \left( {1 + 2\frac{{g_1 g_2 \frac{{\partial
h}}{{\partial x}}}}{{g'_1 g_2 h}}} \right)} \right\}dydx} }
\nonumber\\
&= \int_\alpha ^\beta  {\int_\gamma ^\delta  {\left( {g'_1 g_2 h}
\right)^2 dydx} }  + 2\int_\alpha ^\beta  {\int_\gamma ^\delta
{\left( {g'_1 g_2 h} \right) \left( {g_1 g_2 \frac{{\partial
h}}{{\partial x}}} \right)dydx} }
\nonumber\\
&= \left. {\left. {g_2^2 g_1 \left( {g'_1 } \right) h^2 }
\right|_\gamma ^\delta  } \right|_\alpha ^\beta   - \int_\alpha
^\beta  {\int_\gamma ^\delta {\left( {g_1 g''_1 +(g'_1)^2} \right)
g_2^2 h^2 dydx} }
\nonumber\\
&\qquad+ \int_\alpha ^\beta  {\int_\gamma ^\delta  {\left( {g'_1
g_2 h} \right)^2 dydx} }
\nonumber\\
&= \left. {\left. {g_2^2 g_1 \left( {g'_1 } \right)h^2 }
\right|_\gamma ^\delta  } \right|_\alpha ^\beta   - \int_\alpha
^\beta  {\int_\gamma ^\delta  {\left( { - \omega _1^2 g_1^2  +
\left( {g'_1 } \right)^2 } \right)g_2^2 h^2 dydx} }
\nonumber\\
&\qquad+ \int_\alpha ^\beta  {\int_\gamma ^\delta  {\left( {g'_1
g_2 h} \right)^2 dydx} }
\nonumber\\
&= \left. {\left. {g_2^2 g_1 \left( {g'_1 } \right)h^2 }
\right|_\gamma ^\delta  } \right|_\alpha ^\beta   + \omega _1^2
\int_\alpha ^\beta  {\int_\gamma ^\delta  {g_1^2 g_2^2 h^2 dydx} }
- \int_\alpha ^\beta  {\int_\gamma ^\delta  {\left( {g'_1 }
\right)^2 g_2^2 h^2 dydx} }
\nonumber\\
&\qquad+ \int_\alpha ^\beta  {\int_\gamma ^\delta  {\left( {g'_1
g_2 h} \right)^2 dydx} }
\nonumber\\
&= \left. {\left. {g_2^2 g_1 \left( {g'_1 } \right) h^2 }
\right|_\gamma ^\delta  } \right|_\alpha ^\beta   + \omega _1^2
\int_\alpha ^\beta  {\int_\gamma ^\delta  {g_1^2 g_2^2 h^2 dydx}.
}\label{eq2.4}
\end{align}
Substitute (\ref{eq2.4}) in (\ref{eq2.3}), we get
\begin{align*}
\int_\gamma ^\delta  {\int_\alpha ^\beta  {\left( {\frac{{\partial
^2 f}}{{\partial x\partial y}}} \right)^2 dxdy} }  &\ge \left.
{\left. {\left. {\left. {g_2 \left( {g'_2 } \right) \Phi ^2 }
\right|_\gamma ^\delta  } \right|_\alpha ^\beta   + \omega _2^2
g_2^2 g_1 \left( {g'_1 } \right)h^2 } \right|_\gamma ^\delta  }
\right|_\alpha ^\beta   + \omega _1^2 \omega _2^2 \int_\alpha
^\beta  {\int_\gamma ^\delta  {g_1^2 g_2^2 h^2 dydx} }
\\
&= \left. {\left. {\left. {\left. {g_2 \left( {g'_2 } \right)
\frac{{f_x^2 }}{{g_2^2 }}} \right|_\gamma ^\delta  }
\right|_\alpha ^\beta   + \omega _2^2 g_2^2 g_1 \left( {g'_1 }
\right) \frac{{f^2 }}{{g_1^2 g_2^2 }}} \right|_\gamma ^\delta  }
\right|_\alpha ^\beta   + \omega _1^2 \omega _2^2 \int_\alpha
^\beta  {\int_\gamma ^\delta  {g_1^2 g_2^2 \frac{{f^2 }}{{g_1^2
g_2^2 }}dydx} }
\\
&= \left. {\left. {\left. {\left. {\left( {\frac{{g'_2 }}{{g_2 }}}
\right) f_x^2 } \right|_\gamma ^\delta  } \right|_\alpha ^\beta +
\omega _2^2 \left( {\frac{{g'_1 }}{{g_1 }}} \right)f^2 }
\right|_\gamma ^\delta  } \right|_\alpha ^\beta   + \omega _1^2
\omega _2^2 \int_\alpha ^\beta  {\int_\gamma ^\delta {f^2 dydx} }
\end{align*}
Hence,
\begin{multline}
\int_\alpha ^\beta  {\int_\gamma ^\delta {f^2 dydx} }
\\
\le \frac{1}{\omega _1^2 \omega _2^2 } \int_\gamma ^\delta
{\int_\alpha ^\beta  {\left( {\frac{{\partial ^2 f}}{{\partial
x\partial y}}} \right)^2 dxdy} }-\frac{1}{\omega _1^2 \omega _2^2
}\left. {\left. {\left. {\left. {\left( {\frac{{g'_2 }}{{g_2 }}}
\right)f_x^2 } \right|_\gamma ^\delta  } \right|_\alpha ^\beta
-\frac{1}{\omega _1^2 } \left( {\frac{{g'_1 }}{{g_1 }}} \right)
f^2 } \right|_\gamma ^\delta  } \right|_\alpha ^\beta \label{eq.h}
\end{multline}
Now, since
\begin{align*}
0 \le f^2 \left( {\alpha,\gamma}  \right) = \left(
{\int_c^\gamma\int_a^\alpha {f_{ts}\left( t,s \right)dtds} }
\right)^2 \le \left( {\alpha  - a} \right)\left( {\gamma - c}
\right) \int_c^\gamma\int_a^\alpha {f^2_{ts}\left( t,s
\right)dtds}
\end{align*}
then
\begin{align*}
0 \le f^2 \left( {\alpha,\gamma}  \right)  \le \left( {\alpha  -
a} \right)\left( {\gamma - c} \right) \int_c^\gamma\int_a^\alpha
{f^2_{ts}\left( t,s \right)dtds} \longrightarrow 0,
\end{align*}
as $\alpha \longrightarrow a^ +$ and $\gamma \longrightarrow c^
+$, i.e., $f^2(\alpha,\gamma)=0$ and so that
\begin{align*}
\left\{ {\frac{1}{{\omega_1 ^{2} }}\left( {\frac{{g'_1}}{g_1}}
\right) } \right\} \cdot f^2 \left( \alpha,\gamma  \right)
\longrightarrow 0\,\,\,\,\text{as}\,\,\,\,\alpha  \longrightarrow
a^ +\,\,\text{and}\,\,\gamma \longrightarrow c^ +.
\end{align*}
Similarly,
\begin{align*}
0 \le f_x^2 \left( {\alpha,\gamma}  \right) &= \left(
{\int_c^\gamma {f_{xy}\left( \alpha,y \right)-f_{xy}\left( a,y
\right)dy} } \right)^2
\\
&\le \left( {\gamma - c} \right)\int_c^\gamma {\left(f_{xy}\left(
\alpha,y \right)-f_{xy}\left( a,y \right) \right)^2dy}
\end{align*}
then
\begin{align*}
0 \le f_x^2 \left( {\alpha,\gamma}  \right)  \le \left( {\gamma -
c} \right)\int_c^\gamma {\left(f_{xy}\left( \alpha,y
\right)-f_{xy}\left( a,y \right) \right)^2dy} \longrightarrow 0,
\end{align*}
as $\gamma \longrightarrow c^ +$, i.e., $f^2(\alpha,\gamma)=0$ and
so that
\begin{align*}
\left\{ {\frac{1}{{\omega_1 ^{2} \omega_2 ^{2}}}\left(
{\frac{{g'_2}}{g_2}} \right) } \right\} \cdot f_x^2 \left(
\alpha,\gamma \right) \longrightarrow
0\,\,\,\,\text{as}\,\,\,\,\alpha \longrightarrow a^
+\,\,\text{and}\,\,\gamma \longrightarrow c^ +.
\end{align*}
Then, from (\ref{eq.h}) it follows
\begin{align*}
&\int_\alpha ^\beta  {\int_\gamma ^\delta {f^2 dydx} }
\\
&\le \frac{1}{\omega _1^2 \omega _2^2 } \int_\gamma ^\delta
{\int_\alpha ^\beta  {\left( {\frac{{\partial ^2 f}}{{\partial
x\partial y}}} \right)^2 dxdy} }-\frac{1}{\omega _1^2 \omega _2^2
}\left. {\left. {\left. {\left. {\left( {\frac{{g'_2 }}{{g_2 }}}
\right) f_x^2 } \right|_\gamma ^\delta  } \right|_\alpha ^\beta
-\frac{1}{\omega _1^2 } \left( {\frac{{g'_1 }}{{g_1 }}} \right)
f^2 } \right|_\gamma ^\delta  } \right|_\alpha ^\beta
\\
&\le \frac{1}{\omega _1^2 \omega _2^2 } \int_\gamma ^\delta
{\int_\alpha ^\beta  {\left( {\frac{{\partial ^2 f}}{{\partial
x\partial y}}} \right)^2 dxdy} }
\end{align*}
where $a<\beta<b$ and $c<\delta<d$. Now let $\alpha\longrightarrow
a^{+}$, $\beta\longrightarrow b^{-}$  and $\gamma\longrightarrow
c^{+}$, $\delta\longrightarrow d^{-}$ to obtain the inequality
(\ref{eq.thm1}).

To obtain the sharpness, assume that (\ref{eq.thm1}) holds with
another constant $K>0$,
\begin{align}
\label{eq2.8} \int_c^d{\int_a^b {f^2\left( {x,y} \right)dx}dy} \le
K\left( {b - a} \right)^2 \left( {d - c} \right)^2
\int_c^d{\int_a^b {\left( {\frac{{\partial ^2 f}}{{\partial
x\partial y}}} \right)^2}dxdy}.
\end{align}
Define the function $f:\left[{a,b}\right)\times \left[{c,d}\right)
\to \mathbb{R}$, given by
\begin{align*}
f\left({x,y}\right)=C\sin\left({\frac{\pi}{2}\cdot\frac{x-a}{b-a}}\right)
\sin\left({\frac{\pi}{2}\cdot\frac{y-c}{d-c}}\right),
\end{align*}
therefore, we have $$ \frac{{\partial ^2 f}}{{\partial x\partial
y}} = \frac{{\pi ^2 }}{{4\left( {b - a} \right)\left( {d - c}
\right)}}\cos \left( {\frac{\pi }{2} \cdot \frac{{x - a}}{{b -
a}}} \right)\cos \left( {\frac{\pi }{2} \cdot \frac{{y - c}}{{d -
c}}} \right), $$ $\int_a^b {\int_c^d {f^2 \left( {x,y}
\right)dydx} }  = \frac{{\left( {b - a} \right)\left( {d - c}
\right)}}{4} $, and $\int_a^b {\int_c^d {\left( {\frac{{\partial
^2 f}}{{\partial x\partial y}}} \right)^2 dydx} } =\frac{{\pi ^4
}}{{64\left( {b - a} \right)\left( {d - c} \right)}}$, substitute
in (\ref{eq2.8})
\begin{align*}
\frac{{\left( {b - a} \right)\left( {d - c} \right)}}{4} \le
K\left( {b - a} \right)^2 \left( {d - c} \right)^2\frac{{\pi ^4
}}{{64\left( {b - a} \right)\left( {d - c} \right)}},
\end{align*}
which means that $K\ge \frac{16}{\pi^4}$, thus the constant
$\frac{16}{\pi^4}$ is the best possible and the inequality is
sharp (\ref{eq.thm1}).
\end{proof}

\begin{corollary}
\label{cor1}If $f\in \mathfrak{L}^2(I_-)$, then the inequality
(\ref{eq.thm1}) still holds, and the inequality is sharp.
\end{corollary}

\begin{proof}
The proof goes likewise the proof of Theorem \ref{thm1}, with few
changes in the auxiliary function `$\sin$' in both variable $x$
and $y$ defined on the bidimensional interval $I_-$. To obtain the
sharpness, define the function $f:\left({a,b}\right]\times
\left({c,d}\right] \to \mathbb{R}$, given by
\begin{align*}
f\left({x,y}\right)=C\sin\left({\frac{\pi}{2}\cdot\frac{b-x}{b-a}}\right)
\sin\left({\frac{\pi}{2}\cdot\frac{d-y}{d-c}}\right),
\end{align*}
where $C$ is constant.
\end{proof}

\begin{corollary}
\label{cor2}Let $f \in \mathfrak{L}^2(I)$. Under the assumptions
of Theorem \ref{thm1} and Corollary \ref{cor1} together, the
inequality
\begin{multline}
 \int_c^d {\int_a^b {\left| {f\left( {x,y} \right)
- f\left( {\xi ,\eta } \right)} \right|^2 dxdy} }
\\
\le \frac{{16}}{{\pi ^4 }} \left[ {\frac{{b - a}}{2} + \left| {\xi
- \frac{{a + b}}{2}} \right|} \right]^2 \left[ {\frac{{d - c}}{2}
+ \left| {\eta  - \frac{{c + d}}{2}} \right|} \right]^2
\int_c^d{\int_a^b {\left( {\frac{{\partial ^2 f}}{{\partial
x\partial y}}} \right)^2}dxdy} \label{eq2.8}
\end{multline}
is valid for all $(\xi,\eta) \in \mathbb{D^{\circ}}$. The constant
$\frac{1}{\pi^4}$ is the best possible.
\end{corollary}

\begin{proof}
Applying Theorem \ref{thm1} and Corollary \ref{cor2} on the right
hand side of the equation
\begin{multline*}
\int_c^d {\int_a^b {\left| {f\left( {x,y} \right) - f\left( {\xi
,\eta } \right)} \right|^2 dxdy} }
\\
= \int_c^\eta  {\int_a^{\xi} {\left| {f\left( {x,y} \right) -
f\left( {\xi ,\eta } \right)} \right|^2 dxdy} } + \int_{\eta}^d
{\int_a^{\xi} {\left| {f\left( {x,y} \right) - f\left( {\xi ,\eta
} \right)} \right|^2 dxdy} }
\\
+ \int_c^{\eta}  {\int_{\xi}^b {\left| {f\left( {x,y} \right) -
f\left( {\xi ,\eta } \right)} \right|^2 dxdy} }  + \int_{\eta}^d
{\int_\xi^b {\left| {f\left( {x,y} \right) - f\left( {\xi ,\eta }
\right)} \right|^2 dxdy} }
\end{multline*}
and the make the substitution $h(x,y)= \left| {f\left( {x,y}
\right) - f\left( {\xi ,\eta } \right)} \right|^2$. To obtain the
sharpness define
 $f:\left[{a,b}\right]\times
\left[{c,d}\right] \to \mathbb{R}$, given by
\begin{align*}
f\left( x,y \right) = K_0  + \left\{ \begin{array}{l}
 K_1 \sin  \left( {\frac{{\pi }}{2} \cdot \frac{{a+\xi  - 2x}}{{\xi  - a}}} \right)\sin  \left( {\frac{{\pi }}{2} \cdot \frac{{c+\eta  - 2y}}{{\eta  - c}}} \right)u_\xi \left( \tau \right)u_\eta \left( \psi \right),\,\,\,\,\,\,\,\,\,a \le x \le \xi, c \le y \le \eta  \\
  \\
 K_2 \sin  \left( {\frac{{\pi }}{2} \cdot \frac{{2x - \xi -b}}{{b - \xi }}} \right)\sin  \left( {\frac{{\pi }}{2} \cdot \frac{{c+\eta  - 2y}}{{\eta  - c}}} \right)u_\xi \left( { - \tau } \right)u_\eta \left( \psi \right),\,\,\,\,\,\xi  \le x \le b, c \le y \le \eta \\
\\
 K_3 \sin  \left( {\frac{{\pi }}{2} \cdot \frac{{a+\xi  - 2x}}{{\xi  - a}}} \right)\sin  \left( {\frac{{\pi }}{2} \cdot \frac{{2y - \eta -d}}{{d - \xi }}} \right)u_\xi \left( \tau \right)u_\eta \left( { - \psi } \right),\,\,\,\,\,\,\,a \le x \le \xi, \eta  \le x \le d  \\
  \\
 K_4 \sin  \left( {\frac{{\pi }}{2} \cdot \frac{{2x - \xi -b}}{{b - \xi }}} \right)\sin  \left( {\frac{{\pi }}{2} \cdot \frac{{2y - \eta -d}}{{d - \eta }}} \right)u_\xi \left( { - \tau } \right)u_\eta \left( { - \psi } \right),\,\,\,\,\xi  \le x \le b, \eta  \le x \le d \\
 \end{array} \right.
\end{align*}
where $K_0$, $K_1$, $K_2$, $K_3$ and  $K_2$ are arbitrary
constants, $\tau  = 2\xi  - a - b$, $\psi  = 2\eta  - c - d$ and
$u_s  \left( t \right)$ is the unit step function.
\end{proof}

\section{Sharp bounds for the \v{C}eby\v{s}ev functional}\label{sec3}

The \emph{\v{C}eby\v{s}ev functional}
\begin{align}
\label{eq3.1}\mathcal{T}\left( {f,g} \right) &:= \frac{1}{{\left(
{b - a} \right)\left( {d - c} \right)}}\int_a^b {\int_c^d {f\left(
{t,s} \right)g\left( {t,s} \right)dsdt} }
\nonumber\\
&\qquad- \frac{1}{{\left( {b - a} \right)\left( {d - c}
\right)}}\int_a^b {\int_c^d {f\left( {t,s} \right)dsdt} }
\frac{1}{{\left( {b - a} \right)\left( {d - c} \right)}}\int_a^b
{\int_c^d {g\left( {t,s} \right)dtds} }
\end{align}
has interesting applications in the approximation of the integral
of a product as pointed out in the references below.

In order to represent the remainder of the Taylor formula in an
integral form which will allow a better estimation using the
Gr\"{u}ss type inequalities, Hanna et al. \cite{Hanna4},
generalized the Korkine identity  for double integrals and
therefore Gr\"{u}ss type inequalities were proved.

In 2002, Pachpatte \cite{Pachpatte} has established two
inequalities of Gr\"{u}ss type involving continuous functions of
two independent variables whose first and second partial
derivatives are exist, continuous and belong to
$L_{\infty}(\mathbb{D})$; for details see \cite{Pachpatte}. For
more results about multivariate and multidimensional Gr\"{u}ss
type inequalities the reader may refer to
\cite{Anastassiou1}--\cite{Barnett}, \cite{Hanna1}--\cite{Hanna4}
and \cite{Pachpatte1}.

Recently, the author of this paper \cite{alomari} established
various inequalities of Gr\"{u}ss type for functions of two
variables under various assumptions of the functions involved.

In viewing of Corollary \ref{cor2}, we may state the following
result.
\begin{theorem}
\label{thm5} If $f,g \in \mathfrak{L}^2(\mathbb{D})$, then
\begin{align}
\label{eq3.7}\left| {\mathcal{T}\left( {f,g} \right)} \right|
\le\frac{{1}}{{\pi ^4 }}\left( {b - a} \right)^2 \left( {d - c}
\right)^2 \left\|{\frac{{\partial ^2 f}}{{\partial x\partial
y}}}\right\|_2\left\|{\frac{{\partial ^2 g}}{{\partial x\partial
y}}}\right\|_2,
\end{align}
$\frac{1}{\pi^4}$ is the best possible.
\end{theorem}

\begin{proof}
By the triangle inequality and then using the Cauchy-Schwartz
inequality, we get
\begin{multline}
\left| {\mathcal{T}\left( {f,f} \right)} \right|^2
\\
= \frac{1}{{\Delta^2}}\left| {\int_c^d {\int_a^b {\left[ {f\left(
{x,y} \right) - f\left( {\frac{{a + b}}{2},\frac{{c + d}}{2}}
\right)} \right]\left[ {f\left( {x,y} \right) -
\frac{1}{{\Delta}}\int_c^d {\int_a^b {f\left( {t,s} \right)dtds} }
} \right]} dx}dy } \right|^2.\label{eq3.8}
\\
\le \frac{1}{{\Delta}}\int_c^d {\int_a^b {\left[ {f\left( {x,y}
\right) - f\left( {\frac{{a + b}}{2},\frac{{c + d}}{2}} \right)}
\right]^2 dxdy} }
\\
\times\frac{1}{{\Delta}}\int_c^d {\int_a^b {\left[ {f\left( {x,y}
\right) -\frac{1}{{\Delta}} \int_c^d {\int_a^b {f\left( {t,s}
\right)dtds} } } \right]^2 dxdy} }.
\end{multline}
Now, since
\begin{align*}
\mathcal{T}\left( {f,f} \right)&= \frac{1}{{\Delta}}\int_c^d
{\int_a^b {\left( {f\left( {x,y} \right) -
\frac{1}{{\Delta}}\int_c^d {\int_a^b {f\left( {t,s} \right)dtds} }
} \right)^2 dx}dy}
\\
&= \frac{1}{{\Delta}}\int_c^d \int_a^b  \left[ {f^2\left( {x,y}
\right)-2 f\left( {x,y} \right) \frac{1}{{\Delta}} \int_c^d
{\int_a^b {f\left( {t,s} \right)dt} ds} +\left(
{\frac{1}{{\Delta}}\int_c^d {\int_a^b {f\left( {t,s} \right)dtds}
} } \right)^2} \right]dxdy
\\
&= \frac{1}{{\Delta}}\int_c^d {\int_a^b { f^2\left( {x,y}
\right)}dxdy} - \left( {\frac{1}{{\Delta}}\int_c^d {\int_a^b
{f\left( {t,s} \right)dtds} } } \right)^2,
\end{align*}
where $\Delta:= (b-a)(d-c)$.

Therefore, from (\ref{eq3.8})
\begin{align*}
\left| {\mathcal{T}\left( {f,f} \right)} \right|^2 &\le
\frac{1}{{\Delta}}\int_c^d {\int_a^b {\left[ {f\left( {x,y}
\right) - f\left( {\frac{{a + b}}{2},\frac{{c + d}}{2}} \right)}
\right]^2 dxdy} }
\\
&\qquad\times\frac{1}{{\Delta}}\int_c^d {\int_a^b {\left[ {f\left(
{x,y} \right) - \frac{1}{{\Delta}}\int_c^d {\int_a^b {f\left(
{t,s} \right)dtds} } } \right]^2 dxdy} }
\\
&=\frac{1}{{\Delta}}\int_c^d {\int_a^b {\left[ {f\left( {x,y}
\right) - f\left( {\frac{{a + b}}{2},\frac{{c + d}}{2}} \right)}
\right]^2 dxdy} } \times \mathcal{T}\left( {f,f} \right).
\end{align*}
Therefore,
\begin{align*}
\mathcal{T}\left( {f,f} \right) \le \frac{1}{{\Delta}}\int_c^d
{\int_a^b {\left[ {f\left( {x,y} \right) - f\left( {\frac{{a +
b}}{2},\frac{{c + d}}{2}} \right)} \right]^2 dxdy} }.
\end{align*}
Applying (\ref{eq2.8}), we get
\begin{align*}
\int_c^d {\int_a^b {\left[ {f\left( {x,y} \right) - f\left(
{\frac{{a + b}}{2},\frac{{c + d}}{2}} \right)} \right]^2 dxdy} }
\le \frac{{1}}{{\pi ^4 }} \left( {b - a} \right)^2 \left( {d - c}
\right)^2 \left\|{\frac{{\partial ^2f}}{{\partial x\partial
y}}}\right\|^2_2.
\end{align*}
Thus,
\begin{align}
\mathcal{T}\left( {f,f} \right) \le \frac{{1}}{{\pi ^4 }}\left( {b
- a} \right)^2 \left( {d - c} \right)^2 \left\|{\frac{{\partial ^2
f}}{{\partial x\partial y}}}\right\|^2_2.\label{eq3.9}
\end{align}
In a similar argument we can observe that
\begin{align}
\mathcal{T}\left( {g,g} \right) \le \frac{{1}}{{\pi ^4 }}\left( {b
- a} \right)^2 \left( {d - c} \right)^2 \left\|{\frac{{\partial ^2
g}}{{\partial x\partial y}}}\right\|^2_2.\label{eq3.10}
\end{align}
Finally, since
\begin{align*}
\left|{\mathcal{T}\left( {f,g} \right)}\right| \le
\mathcal{T}^{1/2}\left( {f,f} \right)\mathcal{T}^{1/2}\left(
{g,g}\right) \le\frac{{1}}{{\pi ^4 }}\left( {b - a} \right)^2
\left( {d - c} \right)^2 \left\|{\frac{{\partial ^2 f}}{{\partial
x\partial y}}}\right\|_2\left\|{\frac{{\partial ^2 g}}{{\partial
x\partial y}}}\right\|_2.
\end{align*}
which proves (\ref{eq3.7}). To obtain the sharpness, assume that
(\ref{eq3.7}) holds with another constant $K>0$,
\begin{align}
\label{eq3.11} \left| {\mathcal{T}\left( {f,g} \right)} \right|\le
K\left( {b - a} \right)^2 \left( {d - c} \right)^2
\left\|{\frac{{\partial ^2 f}}{{\partial x\partial
y}}}\right\|_2\left\|{\frac{{\partial ^2 g}}{{\partial x\partial
y}}}\right\|_2.
\end{align}
Define the functions $f,g:\mathbb{D} \to \mathbb{R}$, given by
\begin{align*}
f\left({x,y}\right)=\sin \left( {\frac{\pi }{2} \cdot \frac{{
a+b-2x}}{{b - a}}} \right)\sin \left( {\frac{\pi }{2} \cdot
\frac{{c+d-2y}}{{d - c}}} \right)=g\left({x,y}\right),
\end{align*}
therefore, we have
\begin{align*}
\frac{{\partial ^2 f}}{{\partial x\partial y}} = \frac{{\pi ^2
}}{{\left( {b - a} \right)\left( {d - c} \right)}}\cos \left(
{\frac{\pi }{2} \cdot \frac{{a+b-2x}}{{b - a}}} \right)\cos \left(
{\frac{\pi }{2} \cdot \frac{{c+d-2y }}{{d - c}}}
\right)=\frac{{\partial ^2 g}}{{\partial x\partial y}},
\end{align*}
\begin{align*}
\int_a^b {\int_c^d {f \left( {x,y} \right)g \left( {x,y}
\right)dydx} }  = \frac{{\left( {b - a} \right)\left( {d - c}
\right)}}{4},
\end{align*}
and
\begin{align*}
\int_a^b {\int_c^d {\left( {\frac{{\partial ^2 f}}{{\partial
x\partial y}}} \right)^2 dydx} } = \frac{{\left( {b - a}
\right)\left( {d - c} \right)}}{4}=\int_a^b {\int_c^d {\left(
{\frac{{\partial ^2 g}}{{\partial x\partial y}}} \right)^2 dydx}}.
\end{align*}
Substituting in (\ref{eq3.11})
\begin{align*}
\frac{{\left( {b - a} \right)\left( {d - c} \right)}}{4} \le K
\pi^4\frac{{\left( {b - a} \right)\left( {d - c} \right)}}{4},
\end{align*}
which means that $K\ge \frac{1}{\pi^4}$, thus the constant
$\frac{1}{\pi^4}$ is the best possible and the inequality is sharp
(\ref{eq3.7}).
\end{proof}

\begin{theorem}
\label{thm6}Let $f\in\mathfrak{L}^2(\mathbb{D})$  and $g :
\mathbb{D} \to \mathbb{R}$ satisfies that there exists the real
numbers $\gamma,\Gamma$ such that $\gamma \le g(x,y) \le \Gamma$
for all $(x,y)\in \mathbb{D}$, then
\begin{align}
\label{eq3.12}\left| {\mathcal{T}\left( {f,g} \right)} \right| \le
\frac{4}{{\pi ^2 }}\left( {b - a} \right)^{1/2} \left( {d - c}
\right)^{1/2}\left( {\Gamma - \gamma}
\right)\left\|{\frac{{\partial ^2 f}}{{\partial x\partial
y}}}\right\|_2.
\end{align}
The constant $\frac{4}{\pi^2}$ is the best possible.
\end{theorem}

\begin{proof}
Since
\begin{multline*}
T\left( {f,g} \right)
\\
= \frac{1}{\Delta }\int_a^b {\int_c^d {\left[ {f\left( {x,y}
\right) - \frac{{f\left( {a,c} \right) + f\left( {a,d} \right) +
f\left( {b,c} \right) + f\left( {b,d} \right)}}{4}} \right]\left[
{g\left( {x,y} \right) - \frac{1}{\Delta }\int_a^b {\int_c^d
{g\left( {t,s} \right)} } } \right]dydx} }
\end{multline*}
Taking the absolute value for both sides and making of use  the
triangle inequality, we get
\begin{align}
\label{eq3.13}&\left|{T\left( {f,g} \right)}\right|
\\
&\le \frac{1}{\Delta }\int_a^b {\int_c^d {\left| {f\left( {x,y}
\right) - \frac{{f\left( {a,c} \right) + f\left( {a,d} \right) +
f\left( {b,c} \right) + f\left( {b,d} \right)}}{4}} \right| \left|
{g\left( {x,y} \right) - \frac{1}{\Delta }\int_a^b {\int_c^d
{g\left( {t,s} \right)} } } \right|dydx} }
\nonumber\\
&\le \frac{1}{\Delta}\left( {\int_a^b {\int_c^d {\left| {f\left(
{x,y} \right) - \frac{{f\left( {a,c} \right) + f\left( {a,d}
\right) + f\left( {b,c} \right) + f\left( {b,d} \right)}}{4}}
\right|^2dydx} }} \right)^{1/2}
\nonumber\\
&\qquad\times\left( {\int_a^b {\int_c^d { \left| {g\left( {x,y}
\right) - \frac{1}{\Delta }\int_a^b {\int_c^d {g\left( {t,s}
\right)} } } \right|^2dydx} } }\right)^{1/2}\nonumber
\end{align}
As in Theorem 1 in \cite{alomari}, we have observed that since
there exists $\gamma, \Gamma \ge 0$ such that $\gamma \le g(x,y)
\le \Gamma$ for all $(x,y)\in \mathbb{D}$, then
\begin{align}
\label{eq3.14}\int_c^d {\int_a^b {\left| {g\left( {x,y} \right) -
\frac{1}{{\Delta}}\int_c^d {\int_a^b {g\left( {t,s} \right)dtds} }
} \right|^2dx}dy} \le \frac{1}{4}\left( {\Gamma - \gamma}
\right)^2 \Delta
\end{align}
On the other hand, using the elementary inequality $$\left( {A + B
+ C + D} \right)^2 \le 4\left( {A^2  + B^2  + C^2  + D^2 }
\right),$$ for all $A,B,C,D\ge0$, we also have
\begin{align*}
&\int_a^b {\int_c^d {\left| {f\left( {x,y} \right) -
\frac{{f\left( {a,c} \right) + f\left( {a,d} \right) + f\left(
{b,c} \right) + f\left( {b,d} \right)}}{4}} \right|^2dydx} }
\\
&\le \int_a^b {\int_c^d {\left| {f\left( {x,y} \right) - f\left(
{a,c} \right)} \right|^2dydx} }+ \int_a^b {\int_c^d {\left|
{f\left( {x,y} \right) - f\left( {a,d} \right)} \right|^2dydx} }
\\
&\qquad+\int_a^b {\int_c^d {\left| {f\left( {x,y} \right) -f\left(
{b,c} \right) } \right|^2dydx} }+ \int_a^b {\int_c^d {\left|
{f\left( {x,y} \right) - f\left( {b,d} \right)} \right|^2dydx} }
\end{align*}
Applying (\ref{eq2.8}) for each integral above and simplify we get
\begin{multline}
\label{eq3.15}\int_a^b {\int_c^d {\left| {f\left( {x,y} \right) -
\frac{{f\left( {a,c} \right) + f\left( {a,d} \right) + f\left(
{b,c} \right) + f\left( {b,d} \right)}}{4}} \right|^2dydx} }
\\
\le\frac{{64}}{{\pi ^4 }} \left( {b-a} \right)^2 \left( {d-c}
\right)^2 \int_c^d{\int_a^b {\left( {\frac{{\partial ^2
f}}{{\partial x\partial y}}} \right)^2}dxdy}.
\end{multline}
Combining the inequalities (\ref{eq3.14}) and (\ref{eq3.15}) with
(\ref{eq3.13}) we get the desired result (\ref{eq3.12}).

To prove the sharpness of (\ref{eq3.12}) holds with constant
$C>0$, i.e.,
\begin{align}
\label{eq3.16}\left| {\mathcal{T}\left( {f,g} \right)} \right| \le
C\left( {b - a} \right)^{1/2} \left( {d - c} \right)^{1/2}\left(
{\Gamma - \gamma} \right)\left\|{\frac{{\partial ^2 f}}{{\partial
x\partial y}}}\right\|_2,
\end{align}
and consider the functions $f,g: \mathbb{D} \to \mathbb{R}$ be
defined as
\begin{align*}
f\left( {x,y} \right)&= \sin \left( {\frac{\pi }{2} \cdot \frac{{
a+b-2x}}{{b - a}}} \right)\sin \left( {\frac{\pi }{2} \cdot
\frac{{c+d-2y}}{{d - c}}} \right),
\\
g\left( {x,y} \right) &= {\mathop{\rm sgn}} \left( {x - \frac{{a +
b}}{2}} \right) \cdot {\mathop{\rm sgn}} \left( {y - \frac{{c +
d}}{2}} \right).
\end{align*}
As in the proof of Theorem \ref{thm5}, $f\in
\mathcal{L}^2(\mathbb{D})$, and  $\Gamma - \gamma = 2$, $ \int_a^b
{\int_c^d {g\left( {t,s} \right)dsdt} } = 0$,
\begin{align*}
\int_a^b {\int_c^d {f \left( {x,y} \right)g \left( {x,y}
\right)dydx} }  = \frac{4}{\pi^2}\left( {b - a} \right)\left( {d -
c} \right),
\end{align*}
and
\begin{align*}
\int_a^b {\int_c^d {\left( {\frac{{\partial ^2 f}}{{\partial
x\partial y}}} \right)^2 dydx} } = \frac{{\left( {b - a}
\right)\left( {d - c} \right)}}{4}.
\end{align*}
Making use of (\ref{eq3.16}) we get $\frac{4}{\pi^2} \le C$, which
proves that $\frac{4}{\pi^2} $ is the best possible and thus the
proof is completely finished.
\end{proof}

\subsection{An inequality of Ostrowski's type} \label{sec4}
The mean value theorem for double integrals reads that: If $f$ is
continuous on $[a,b]\times[c,d]$, then there exists
$(\eta,\xi)\in[a,b]\times[c,d]$ such that
\begin{align}
f\left( {\eta,\xi} \right) = \frac{1}{{\left( {b - a}
\right)\left( {d - c} \right)}}\int_c^d {\int_a^b {f\left( {t,s}
\right)dtds} }.
\end{align}

What about if one needs to measure the difference between the
image of an arbitrary point $(x,y)\in[a,b]\times[c,d]$ and the
average value $\frac{1}{{\left( {b - a} \right)\left( {d - c}
\right)}}\int_c^d {\int_a^b {f\left( {t,s} \right)dtds} }$?

In this way Ostrowski introduced his famous inequality regarding
differentiable functions and its average values. In
\cite{Dragomir} and \cite{Hanna1}--\cite{Hanna4} and other related
works many authors have studied the Ostrowski's type inequalities
for various type of functions of several variables.

In the following, we present a bound belongs to $L_2$ norm for the
Ostrowski inequality.

\begin{theorem}
\label{thm9}Let $f\in\mathfrak{L}^2(\mathbb{D})$ , then
\begin{multline}
\label{eq4.2}\left| {f\left( x,y \right)- \frac{1}{{\left( b-a
\right)\left( d-c \right)}} \int_a^b {\int_c^d{f\left( t,s
\right)dt}ds} } \right|
\\
\le \frac{{4}}{{\pi ^2\Delta^{1/2}}} \left[ {\frac{{b - a}}{2} +
\left| {x - \frac{{a + b}}{2}} \right|} \right] \left[ {\frac{{d -
c}}{2} + \left| {y  - \frac{{c + d}}{2}} \right|} \right]
\left\|{\frac{{\partial ^2 f}}{{\partial t\partial s}}}\right\|_2
\end{multline}
for all $(x,y) \in [a,b]\times[c,d]$. In special case, choose
$(x,y)=\left({\frac{a+b}{2},\frac{c+d}{2}}\right)$
\begin{align*}
\left| {f\left({\frac{a+b}{2},\frac{c+d}{2}}\right)-
\frac{1}{{\left( b-a \right)\left( d-c \right)}} \int_a^b
{\int_c^d{f\left( t,s \right)dt}ds} } \right| \le \frac{{1}}{{\pi
^2}}\Delta^{1/2} \left\|{\frac{{\partial ^2 f}}{{\partial
t\partial s}}}\right\|_2
\end{align*}

\end{theorem}

\begin{proof}
Since
\begin{align*}
f\left( x,y \right)- \frac{1}{{\left( b-a \right)\left( d-c
\right)}} \int_a^b {\int_c^d{f\left( t,s \right)dt}ds} =
\frac{1}{{\left( b-a \right)\left( d-c \right)}} \int_a^b
{\int_c^d{ \left[{f\left( x,y \right)-f\left( t,s
\right)}\right]dt}ds}
\end{align*}
Taking the modulus, applying the triangle inequality and then use
the Cauchy-Schwarz's inequality, we get
\begin{align*}
&\left| {f\left( {x,y} \right) - \frac{1}{{\Delta}}\int_a^b
{\int_c^d {f\left( {t,s} \right)dtds} } } \right|
\\
&\le\frac{1}{\Delta }\int_a^b {\int_c^d {\left| {f\left( {x,y}
\right) - f\left( {t,s} \right)} \right|dtds} }
\nonumber\\
&\le \frac{1}{\Delta^{1/2} }\left({\int_a^b {\int_c^d {\left|
{f\left( {x,y} \right) - f\left( {t,s} \right)} \right|^2dtds}
}}\right)^{1/2}
\nonumber\\
&\le \frac{{4}}{{\pi ^2\Delta^{1/2}}} \left[ {\frac{{b - a}}{2} +
\left| {x - \frac{{a + b}}{2}} \right|} \right] \left[ {\frac{{d -
c}}{2} + \left| {y  - \frac{{c + d}}{2}} \right|} \right]
\left\|{\frac{{\partial ^2 f}}{{\partial t\partial s}}}\right\|_2,
\end{align*}
which follows by (\ref{eq2.8}), and this proves (\ref{eq4.2}).
\end{proof}

\noindent\textbf{Conflict of Interests.} The author declares that
there is no conflict of interests regarding the publication of
this paper.

\centerline{}

\centerline{}

\end{document}